\documentclass[12pt,a4paper]{amsart}
\usepackage{graphics}
\usepackage{epsfig}
\usepackage[all]{xy}
\usepackage{graphicx}
\theoremstyle{plain}
\usepackage{amssymb}
\usepackage[english]{babel}
\advance\hoffset-20mm \advance\textwidth40mm

\newtheorem{theorem}{Theorem}
\newtheorem{lemma}{Lemma}
\newtheorem*{theo*}{Theorem}

\newtheorem{corollary}{Corollary}
\theoremstyle{definition}

\newtheorem*{definition*}{Definition}

\newtheorem{remark}{Remark}

\usepackage[hyperindex,bookmarksnumbered,unicode]{hyperref}

\begin{document}
	\sloppy
	\title[On Clifford Algebras of Infinite Dimensional Vector Spaces]
	{On Clifford Algebras of Infinite Dimensional Vector Spaces}
	\author
	{Oksana Bezushchak}
	\address{Oksana Bezushchak: Faculty of Mechanics and Mathematics, Taras Shevchenko National University of Kyiv, Volodymyrska, 60, Kyiv 01033, Ukraine}
	\email{bezushchak@knu.ua}
	
	\date{August 14, 2024.}

	\maketitle
	%%%%%%%%%%%%%%%%%%%%%%%%%%%%%%%%%%%%%%%%%%%%%%%%%%%%%%%%%%%
	
	\begin{abstract}
		We describe  derivations  of the Clifford algebra of a nondegenerate quadratic form on a countable dimensional vector space over an algebraically closed field of characteristic not equal to $2$. We also construct an algebraic  automorphism of the Clifford algebra of a positive definite quadratic form  that is not continuous. 	
	\end{abstract}

	%%%%%%%%%%%%%%%%%%%%%%%%%%%%%%%%%%%%%%%%%%%%%%%%%%%%%%%%%%%

	\section*{Introduction}

Let $\mathbb{F}$ be a field of characteristic not equal to $2$. Let $V$ be a vector space over $\mathbb{F}$. A mapping $f: V \times V \rightarrow \mathbb{F}$ is called a \emph{quadratic form} if

(1) $f(\lambda v) = \lambda^2 f(v)$,

(2) $f(v, w) = f(v+w) - f(v) - f(w)$ is a bilinear form.

A quadratic form $f$ is \emph{nondegenerate} if the bilinear form $f(v, w)$ is nondegenerate.

The \emph{Clifford algebra} $\mathcal{C}\ell(V,f)$ is generated by the vector space $V$ and unit $1$ with defining relations $v^2 = f(v)\cdot 1$ for $v \in V$. If $\{v_i\}_{i\in I}$ is a basis of the vector space $V$ and the set of indices $I$ is ordered, then the set of ordered products $ v_{i_1}\cdots v_{i_k} , $ where $ i_1 < i_2 < \ldots < i_k,$ and $1$ (viewed as the empty product), form a basis of the Clifford algebra $\mathcal{C}\ell(V,f)$.

The Clifford algebra $\mathcal{C}\ell(V,f)$ is graded by the cyclic group of order $2$, expressed as: $$\mathcal{C}\ell(V,f) = \mathcal{C}\ell(V,f)_{\overline{0}} + \mathcal{C}\ell(V,f)_{\overline{1}}, $$ where $$ \mathcal{C}\ell(V,f)_{\overline{0}} = \mathbb{F}\cdot 1 + \sum_{n = 1}^{\infty} \underbrace{V \cdots V}_{2n}, \quad \mathcal{C}\ell(V,f)_{\overline{1}} =  \sum_{n = 0}^{\infty} \underbrace{V \cdots V}_{2n+1}.$$

Suppose that  the ground field $\mathbb{F}$ is algebraically closed, and the quadratic form is nondegenerate. If $\dim_{\mathbb{F}} V =d$ is an even integer,  then the Clifford algebra $\mathcal{C}\ell(V,f)$ is isomorphic to the algebra $M_{2^{\frac{d}{2}}} (\mathbb{F})$ of $2^{\frac{d}{2}} \times 2^{\frac{d}{2}}$ matrices over $\mathbb{F}$. If $d$ is odd, then  $$\mathcal{C}\ell(V,f) \cong M_{2^{\frac{d-1}{2}}} (\mathbb{F}) \oplus M_{2^{\frac{d-1}{2}}} (\mathbb{F});$$ see \cite{Jacobson_1}.

An algebra $A$ over a field $\mathbb{F}$ is called a \emph{locally matrix algebra} if an arbitrary finite collection of elements of $A$ lies  in a subalgebra $B \leq A$ such that  $B$ is isomorphic to the algebra $M_n(\mathbb{F})$ of $n \times n$ matrices over $\mathbb{F}$ for some $n \geq 1$. An algebra $A$ is \emph{unital} if it contains the unit $1$. 

In \cite{BezOl}, we proved  that for an infinite dimensional vector space $V$ over an algebraically closed field $\mathbb{F}$ with a nondegenerate quadratic form $f:V  \rightarrow \mathbb{\mathbb{F}}$, the Clifford algebra $\mathcal{C}\ell(V,f)$ is a unital locally matrix algebra.

G.~K\"{o}the \cite{Koethe} proved that a countable dimensional unital locally matrix algebra is isomorphic to a tensor product of matrix algebras. For locally matrix algebras of uncountable dimension, this is no longer true; see~\cite{BezOl_2,Kurosh}.

We begin with an explicit decomposition of the Clifford algebra $\mathcal{C}\ell(V,f)$, where $\dim_{\mathbb{F}} V=\aleph_0$, into a tensor product of matrix algebras; see Sec.~\ref{Sec1_1}. Given that the vector space is countably dimensional, we can, without loss of generality, assume that the set $I$ is the set of positive integers  $\mathbb{N}.$

Let $v_i,$ $i \in \mathbb{N},$ be an orthonormal basis of the space $V$, where $v_i v_j + v_j v_i = 2\delta_{ij}$ (here, $ \delta_{ij}$    is the Kronecker delta);  let $0=n_0< n_1 <  \cdots$ be an increasing sequence of even numbers. Consider the subspace $V_i$ for  $i \in \mathbb{N},$  spanned by $v_{n_{i-1}+1},\ldots , v_{n_i}$. The subalgebra of $\mathcal{C}\ell(V,f)$ generated by  $V_i$ is isomorphic to $\mathcal{C}\ell(V_i,f)$. The algebra $\mathcal{C}\ell(V_i,f)$ is  $\mathbb{Z}/ 2\mathbb{Z}$-graded:  $$\mathcal{C}\ell(V_i,f)=\mathcal{C}\ell(V_i,f)_{\overline{0}}+\mathcal{C}\ell(V_i,f)_{\overline{1}}.$$ Let $c_i = v_1 \cdots v_{n_i}$. Now, consider the sequence of subalgebras: $$A_1=\mathcal{C}\ell(V_1,f), \quad A_i=\mathcal{C}\ell(V_i,f)_{\overline{0}}+c_i\,  \mathcal{C}\ell(V_i,f)_{\overline{1}} $$ of the algebra $\mathcal{C}\ell(V,f)$.

\begin{theorem}\label{TH_1}  $A_i\cong \mathcal{C}\ell(V_i,f)$ \   for each   \  $i\in \mathbb{N}$; \ \ \  $[A_i, A_j] = (0)$ for $i,j\in N,$ $i \neq j$; \ \ \ and $$\mathcal{C}\ell(V,f) \cong \bigotimes_{i \in \mathbb{N}} A_i.$$ \end{theorem}

In this paper, we discuss derivations and automorphisms of the Clifford algebra of an infinite dimensional vector space $V$ endowed  with a nondegenerate quadratic form. There are two known families of examples: 

\begin{enumerate}
	\item[1)] inner derivations and inner automorphisms;
	\item[2)] Bogolyubov derivations and Bogolyubov automorphisms. 
\end{enumerate}

Let $\varphi$ be an invertible linear transformation  $\varphi: V \rightarrow V$  that preserves the quadratic form,  $f( \varphi(v))= f(v)$ for an arbitrary element $v \in V$. It is easy to see that  $\varphi$ uniquely extends to an automorphism of the algebra $\mathcal{C}\ell(V,f)$. Such automorphisms are called \emph{Bogolyubov automorphisms}.

Let $\psi: V \rightarrow V$ be a skew-symmetric linear transformation, that is, $$f(\psi(v), w) + f(v, \psi(w)) = 0 \quad \text{for all} \quad v, w \in V.$$ The mapping $\psi$ uniquely extends to a derivation of the algebra $\mathcal{C}\ell(V,f)$. These derivations are called \emph{Bogolyubov derivations}.

In Sec.~\ref{Sec1}, we use the techniques from \cite{14} to describe  derivations of $\mathcal{C}\ell(V,f)$.

Let $S = \{ i_1< \cdots < i_r \}$ be a finite set of positive integers. Denote $v_S = v_{i_1} \cdots v_{i_r}$.

A \emph{derivation} $D$ of the algebra $\mathcal{C}\ell(V,f)$ is called \emph{even} if $$D(\mathcal{C}\ell(V,f)_{\overline{0}}) \subseteq \mathcal{C}\ell(V,f)_{\overline{0}}, \quad D(\mathcal{C}\ell(V,f)_{\overline{1}}) \subseteq \mathcal{C}\ell(V,f)_{\overline{1}};$$ and a \emph{derivation} $D$ of the algebra $\mathcal{C}\ell(V,f)$ is called \emph{odd} if $$D(\mathcal{C}\ell(V,f)_{\overline{0}}) \subseteq \mathcal{C}\ell(V,f)_{\overline{1}}, \quad D(\mathcal{C}\ell(V,f)_{\overline{1}}) \subseteq \mathcal{C}\ell(V,f)_{\overline{0}}.$$

Note  that  Bogolyubov derivations of Clifford algebras are even.

\begin{theorem}\label{TH2} 	Let $V$ be a countable dimensional vector space  over an algebraically closed field $\mathbb{F}$, and let $v_i,$  $i\in \mathbb{N}$, be an arbitrary orthonormal basis of the space $V$.
\begin{enumerate}
\item[\emph{$(1)$}] Any nonzero even derivation $D$ of $\mathcal{C}\ell(V,f)$ can be uniquely represented as a sum 
\begin{equation}\label{sum1} D = \sum_{S} \alpha_S \, \textit{\emph{ad}}(v_S), \quad  \quad 0\neq\alpha_S \in \mathbb{F},
\end{equation} where the subsets $S$ are finite nonempty  subsets of  $\mathbb{N}$ of even order, and any $i \in \mathbb{N}$ lies in no more than finitely many subsets $S$.
			
\item[\emph{$(2)$}] Any nonzero odd derivation $D$ of $\mathcal{C}\ell(V,f)$ can be uniquely represented as a sum 	
\begin{equation}\label{sum2} D = \sum_{S} \alpha_S \, \textit{\emph{ad}}(v_S), \quad  \quad 0\neq\alpha_S \in \mathbb{F},
\end{equation}
where the subsets $S$ are finite subsets of  $\mathbb{N}$ of odd order, and any $i\in \mathbb{N}$ lies in all but finitely many subsets $S$.	
\end{enumerate}	
\end{theorem}

Notice that a derivation is inner if and only if the sum (\ref{sum1}) or (\ref{sum2}) is finite.

\begin{remark} From the paper \cite{14}, it follows that the dimension of  the Lie algebra $\emph{\text{Der}}(\mathcal{C}\ell(V,f))$, which consists of  all derivations of the Clifford algebra $\mathcal{C}\ell(V,f)$, is    $|\mathbb{F}|^{\dim _{\mathbb{F}}\,V}.$ 
\end{remark}

\begin{theorem}\label{TH3} Let $V$ be a countable dimensional vector space  over an algebraically closed field $\mathbb{F}$, and let $v_i,$ $i \in \mathbb{N},$ be an arbitrary orthonormal basis of the space $V$. A nonzero  even derivation $D $ of $ \mathcal{C}\ell(V,f) $ is a Bogolyubov derivation if and only if 
	\begin{equation}\label{EQ_3} D = \sum_{i<j} \alpha_{ij}\, \textit{\emph{ad}}(v_i v_j), \quad \alpha_{ij}\in \mathbb{F}, 	\end{equation}  where  for each $ i \in \mathbb{N} $ only finitely many coefficients $ \alpha_{ij},$ $i<j,$ are nonzero. Moreover, this  presentation~\emph{(\ref{EQ_3})}  is unique. \end{theorem}

For a finite dimensional vector space $ V $, Theorem \ref{TH3} is discussed in \cite{jacobson}.

\begin{corollary}\label{Cor1}  The Bogolyubov derivation $ D $ corresponding to a skew-symmetric linear transformation $ \psi $ of a countable dimensional vector space $V$ over an algebraically closed field $\mathbb{F}$ is inner if and only if $ \psi$  is finitary, that is, $\dim_\mathbb{F} \psi(V) < \infty.$ \end{corollary}

P. de la Harpe \cite{harpe}  proved that a Bogolyubov automorphism  corresponding to an orthogonal linear transformation $ \varphi $  of $ V $  is inner if and only if the transformation $ \varphi $  is finitary, that is, $\text{Id}-\varphi $ has finite dimensional image; here $\text{Id}$ is the identity transformation.

Suppose that $\mathbb{F} = \mathbb{R}$ is the field of real numbers, and let $f: V \to \mathbb{R}$ be a positive definite  quadratic form. Then the Clifford algebra $\mathcal{C}\ell(V,f) $ possesses  a natural structure of a normed algebra; see Sec.~\ref{Sec2}. M.~Ludewig \cite{Ludewig} raised the question of whether every automorphism of $\mathcal{C}\ell(V,f)$ is continuous with respect to this norm.

In Sec.~\ref{Sec2}, we construct an algebraic  automorphism of $\mathcal{C}\ell(V,f)$ that is not continuous in this norm.

For more information about Clifford algebras of infinite dimensional vector spaces, see~\cite{harpe,shale_stinespring,Wene}.

\section{Tensor decompositions}\label{Sec1_1}

\begin{proof}[Proof of Theorem \ref{TH_1}] Since the elements $n_i$ are even, it follows that $c_i^2 = -1$, the element $c_i$ commutes with all elements from $\mathcal{C}\ell(V_i,f)_{\overline{0}}$ and anticommutes with all elements from $\mathcal{C}\ell(V_i,f)_{\overline{1}}$. The mapping $\varphi_i: \mathcal{C}\ell(V_i,f)  \mapsto A_i;$ defined by $$\varphi_i(u)=u, \ u\in \mathcal{C}\ell(V_i,f)_{\overline{0}}; \quad \varphi_i(u)=c_i u, \ u\in \mathcal{C}\ell(V_i,f)_{\overline{1}},$$  is an isomorphism.
Let us  show that arbitrary elements $a \in A_i$ and  $b \in A_j$, where $i \neq j$, commute. If $a \in \mathcal{C}\ell(V_i,f)_{\overline{0}}$, then $a$ commutes with all elements $v_{k}$, where  $n_{j-1} +1 \leq k \leq  n_j$, and $a$ commutes with $c_j$. Thus, $[a, A_j] = 0$.

Denote  $c_i'=c_i v_{n_i}=v_1 \cdots v_{n_i -1}$. Then $$c_i \, \mathcal{C}\ell(V_i,f)_{\overline{1}} = c_i' \, \mathcal{C}\ell(V_i,f)_{\overline{0}}.$$
It remains to show that the element $c_i'$ commutes with the element $c_j'$. We have $c_j' = c_i'\, v_{n_i} v_{n_i + 1} \cdots v_{n_j -1}$. The element $v_{n_i} v_{n_i+1} \cdots v_{n_j -1}$ has an even length and commutes with elements  $v_{1},$ $\ldots,$ $ v_{n_i-1}$. Hence, it commutes with $c_i'$, completing the proof that the elements $c_i'$ and  $c_j'$ commute, and therefore, $[A_i, A_j] = (0)$.

Finally, we show that $\sum_{i\in \mathbb{N}} A_i$ generates the algebra $\mathcal{C}\ell(V,f)$. Let $k \in \mathbb{N}$, $n_{i-1} < k \leq n_{i}$. If $i=1,$  then the element $v_k$ lies in $A_1 = \mathcal{C}\ell(V_1)$. If $i\geq 2$, using induction on $i$, we assume that $$v_1, \  \ldots, \ v_{n_{i-1}} \in A_1 A_2 \cdots A_{i-1}.$$ We have $$c_i = c_{i-1} \, v_{n_{i-1}+1} \cdots v_{n_i}, \quad v_{n_{i-1}+1} \cdots v_{n_i}= \pm v_k \, u,$$ where $u = v_{n_{i-1}+1} \cdots \widehat{v_k} \cdots  v_{n_i}$ and $u$ lies in $\mathcal{C}\ell(V_i,f)_{\overline{1}}.$ Hence, $$v_k \in v_{n_{i-1}} \cdots v_{n_i}\,  \mathcal{C}\ell(V_i,f)_{\overline{1}} \subseteq  c_{i-1} c_i \, \mathcal{C}\ell(V_i,f)_{\overline{1}} \subseteq A_1 \cdots A_i$$ by the induction assumption. This completes the proof of the theorem. \end{proof}

\section{Derivations}\label{Sec1}

Let us review the description of derivations of infinite tensor products of matrix algebras from \cite{14}.

Let $$A = \bigotimes_{i = 1}^{\infty} A_i, \quad A_i \cong M_{n_i}(\mathbb{F}).$$ A system $\wp$ of nonempty finite subsets of $\mathbb{N}$ is said to be \emph{sparse} if
\begin{enumerate}
	\item[(1)] for any $S \in \wp$ all nonempty subsets of $S$ also lie in $\wp$,
	\item[(2)] each  $i \in \mathbb{N}$  is included in no more than finitely many subsets from $\wp$.
\end{enumerate}
For a subset $S = \{i_1, \ldots, i_r\}$ of  $\mathbb{N}$, denote $$A_S = A_{i_1} \bigotimes \cdots \bigotimes A_{i_r} \cong M_{n_{i_1} \cdots n_{i_r}}(\mathbb{F}).$$ Let $\wp$ be a sparse system. For each subset $S \in \wp$, choose an element $a_S \in A_S$. The sum
\begin{equation}\label{EQ_1}
 \sum_{S \in \wp} \text{ad}(a_S)	
\end{equation}
converges in the Tykhonoff topology to a derivation of $A$. Indeed, choose an arbitrary element $a \in A$. Let $$a \in A_{i_1} \bigotimes \cdots \bigotimes A_{i_r}.$$ Due to the sparsity of the system $\wp$, for all but finitely many subsets $S \in \wp$, we have $\{i_1, \ldots, i_r\} \cap S = \emptyset$, and $\text{ad}(a_S) a = 0$. Let $D_{\wp}$ be the vector space of all such sums (\ref{EQ_1}), where  $D_{\wp} \subseteq \text{Der}(A)$. In \cite{14}, we proved that \begin{equation}\label{EQ_2}
	\text{Der}(A) = \bigcup_{\wp} D_{\wp}, \end{equation}
where the union is taken over all sparse systems of subsets of $\mathbb{N}$.

\begin{proof}[Proof of Theorem \ref{TH2}] Let \(D\) be a derivation of \(\mathcal{C}\ell(V,f)\). By Theorem~2 of \cite{14}, there exists a sparse system of finite subsets \(\wp = \{S \subset \mathbb{N}\}\) and elements \(a_S \in A_S\) such that \(D = \sum_{S \in \wp}  \text{ad}(a_S)\).
	
For an arbitrary set \(S = \{i_1 < \cdots < i_r\} \in \wp\), define $$ \min S   = n_{i_1 -1} + 1, \quad \max S = n_{i_r}.$$ Let \(D\) be an even derivation. Then \(a_S \in (A_S)_{\overline{0}}\) for all \(S \in \wp\). The subalgebra \((A_S)_{\overline{0}}\) is contained in the span of products \(v_{j_1} \cdots v_{j_l}\), where  \( \min S \leq j_1 < \cdots < j_l \leq \max S \) and \(l\) is an even number. Consequently, $$a_S = \sum_{S, \ \min S \leq j_1 < \cdots < j_l \leq \max S, \ l \text{ is even}} \alpha_{j_1  \cdots  j_l} \ v_{j_1} \cdots v_{j_l}, \quad  \alpha_{j_1  \cdots  j_l} \in \mathbb{F}.$$
This formulation yields a presentation $$D = \sum_{S, \ \min S \leq j_1 < \cdots < j_l \leq \max S, \ l \text{ is even}} \alpha_{j_1  \cdots  j_l} \ \text{ad}(v_{j_1} \cdots v_{j_l}), \quad  0\not= \alpha_{j_1  \cdots  j_l} \in \mathbb{F}. $$
For any \(i \in \mathbb{N}\), the number \(i\) is not contained in \(\{j_1, \ldots , j_l\}\) as soon as \(i< \min S\). Hence, \(v_i\) is included in finitely many products \(v_{j_1} \cdots v_{j_l}\).

If  \(D\) is an odd derivation, then \(a_S \in (A_S)_{\overline{1}}\). The space \((A_S)_{\overline{1}}\) is contained in $$c_{j_1} \cdot \text{Span} \Big(v_{j_1} \cdots v_{j_l}, \quad  l \text{ is odd}, \quad \min S \leq j_1 < \cdots < j_l \leq \max S \Big).$$
This results in the presentation $$D = \sum_{S, \ \min S \leq j_1 < \cdots < j_l \leq \max S, \ l \text{ is odd}} \alpha_{j_1  \cdots  j_l} \ \text{ad}(c_{i_1}\, v_{j_1} \cdots v_{j_l}), \quad  0\not= \alpha_{j_1  \cdots  j_l} \in \mathbb{F}. $$ A generator \(v_i\) is involved in \(c_{i_1} v_{j_1} \cdots v_{j_l}\) as soon as \(i < \max S\), hence \(v_i\) is involved in all but finitely many  summands.

Now, let's prove uniqueness. To demonstrate this, let us show that
$$\sum_{S} \alpha_S \, \text{ad}(v_S) = 0 $$  implies $ \alpha_S = 0 $ for all $S$.

For an even derivation, let $ k \in \mathbb{N}$, and let $ S_1,$ $\ldots,$ $S_r $ be all finite  subsets of even orders from the sum that contain $ k. $ Then $v_k v_{S_i}=-v_{S_i} v_k,$ leading to $$ \left( \sum_{S} \alpha_{S}\, \text{ad}(v_{S})\right) \big(v_k\big) = \left(\sum_{i=1}^r \alpha_{S_i} \, \text{ad}(v_{S_i})\right) \big(v_k\big) =2 v_k \sum_{i=1}^r \alpha_{S_i}\, v_{S_i} = 0, $$ and therefore,  $\alpha_{S_1} = \cdots = \alpha_{S_r} =0 $.

For an odd derivation, let \(S_1, \ldots, S_r\) be all finite subsets of odd orders from the sum that do not contain \(k\). Following a similar argument as for even derivations, \(v_k v_{S_i} = -v_{S_i}  v_k, 1 \leq i \leq r\), and if \(k \in  S\), then \(v_k v_S = v_S v_k\), which results in 
\[ \left( \sum_{S} \alpha_S \ \text{ad}(v_S)\right) \big(v_k\big) =  \left( \sum_{i=1}^{r} \alpha_{S_i}\  \mathrm{ad}(v_{S_i})\right) \big(v_k\big) = 2\, v_k \ \sum_{i=1}^{r} \alpha_{S_i} v_{S_i}  = 0 ,\]
thus, \(\alpha_{S_1} =\cdots =\alpha_{S_r} = 0\).

This completes the proof of the theorem. \end{proof}

\begin{proof}[Proof of Theorem \ref{TH3}] Let $ D = \sum_{S} \alpha_S \, \text{ad}(v_S)$  be a Bogolyubov derivation. Choose $k \in \mathbb{N}.$ As mentioned previously, let $S_1,$ $ \ldots,$ $S_r$  be all finite  subsets of even orders   that contain $ k.$ Let $ S'_i = S_i \backslash \{k\}$ for $ 1 \leq i \leq r.$  Then $v_k v_{S_i} = \pm v_{S_i'}.$ We have $$  \left( \sum_{S} \alpha_S \, \text{ad}(v_S) \right)  \big(v_k\big) =  \left( \sum_{i=1}^r \alpha_{S_i} \, \text{ad}(v_{S_i}) \right) \big(v_k\big) = 2 v_k \sum_{i=1}^r \alpha_{S_i}\, v_{S_i}=  2 \sum_{i=1}^r \pm\alpha_{S_i}\, v_{S'_i} \in V.$$  This implies $ |S'_i| = 1$  for $ 1 \leq i \leq r, $  and $ |S_i| = 2. $ 
	
On the other hand, if an even derivation $ D $ has a required presentation $$ D = \sum_{i<j} \alpha_{ij} \, \text{ad}(v_i v_j), $$ then $ D$  maps $ V $ to $ V, $  and therefore, $ D$   is a Bogolyubov derivation. This completes the proof of the theorem. \end{proof}

\begin{proof}[Proof Corollary \ref{Cor1}] Theorems \ref{TH2}, \ref{TH3} imply that if a Bogolyubov derivation $ D $ is inner, then the sum $$  \sum_{i<j} \alpha_{ij} \, \text{ad}(v_i v_j)$$ is finite, and  $$ D = \sum_{1 \leq i<j \leq n} \alpha_{ij} \, \text{ad}(v_i v_j) \quad \text{ for some } \quad n\geq 2.$$ Consequently, $\psi(V) \subseteq \mathrm{Span}(v_1, \ldots, v_n). $ 
	
On the other hand, if $ \dim_\mathbb{F} \psi(V)  < \infty ,$ then $\psi(V) \subseteq \mathrm{Span}(v_1, \ldots, v_n) $ for some $ n;$ and $ D = \sum_{1 \leq i<j \leq n} \alpha_{ij}\, \mathrm{ad}(v_i v_j). $ This completes the proof of the corollary. \end{proof}

\section{Automorphisms of Clifford algebras}\label{Sec2} 

We will start with a trace construction that applies to all unital locally matrix algebras. Let \( A \) be a unital locally matrix algebra over a field \( \mathbb{F} \). Define: \[ \mathcal{D}(A) = \{ n \in \mathbb{N} \ | \text{ there exists a subalgebra }   B \leq A, \text{ with } 1\in B \text{ and }   B \cong M_n(\mathbb{F}) \}. \]

\begin{lemma}\label{lem1}  Suppose that the characteristic of the ground field \( \mathbb{F} \) is coprime with all integers from \( \mathcal{D}(A) \). Then there exists a unique linear functional \( \emph{\text{tr}} \colon A \to \mathbb{F} \) such that \( \emph{\text{tr}}(ab) = \emph{\text{tr}}(ba) \) for arbitrary elements \( a, b \in A \), and \( \emph{\text{tr}}(1) = 1 \). For an arbitrary matrix subalgebra \(  B \leq A \) where $1 \in B$ and \( B \cong M_n(\mathbb{F}) \), the restriction of the functional \( \emph{\text{tr}} \) to \( B \) coincides with the normalized matrix trace.\end{lemma}
\begin{proof} Let \( 1 \in B \leq C \leq A \) be matrix subalgebras of \( A \), where  \( B \cong M_n(\mathbb{F}) \) and \( C \cong M_m(\mathbb{F}) \). Then \( k = m/ n \) is an integer and the embedding \( B \hookrightarrow C \) is diagonal
$$ a \rightarrow \text{diag}(\underbrace{a, a, \dots, a}_k ).$$ This implies that \( \text{tr}_C(a) = k \cdot \text{tr}_B(a) \), and therefore, $$ \frac{1}{n} \  \text{tr}_B(a) = \frac{1}{m} \ \text{tr}_C(a) .$$ 

Now, for any element \( a \in A \), choose a subalgebra \( B \leq A \) such that \( 1, a \in B \) and \( B \cong M_n(\mathbb{F}) \). Define: $$ \text{tr}(a) = \frac{1}{n} \ \text{tr}_B(a) .$$ In view of the above, \( \text{tr}(a) \) does not depend on the choice of the subalgebra \( B \). 

The uniqueness of the trace functional follows from the uniqueness of the trace functional on a matrix algebra. This completes the proof of the lemma. \end{proof}

Let \( \mathbb{F} = \mathbb{R} \) be the field of real numbers, and let \( f: V \rightarrow \mathbb{R} \) be a positive definite quadratic form on a countable dimensional vector space \( V \).

Since the form \( f \) is positive definite, it follows that for any subspace \( W \subset V \) of even dimension \( 2k \), the  Clifford algebra \( \mathcal{C}\ell(W,f) \) is isomorphic to the matrix algebra \( M_{2^{k}}(\mathbb{R}) \). Consequently, the algebra \( \mathcal{C}\ell(V,f) \) is a unital locally matrix algebra over \( \mathbb{R} \).

In \cite{BezOl}, it is shown that \[ \mathcal{D}(\mathcal{C}\ell(V,f)) = \{2^n,\ n \geq 0\}. \] Hence, by Lemma \ref{lem1}, there exists a unique trace functional $ \text{tr} \colon \mathcal{C}\ell(V,f) \to \mathbb{R}. $

\begin{remark} In \cite{Bezushchak_Isom}, we introduced determinants on  unital locally matrix algebras. \end{remark}

The Clifford algebra \( \mathcal{C}\ell(V,f) \) is equipped with an involution \( \ast \) that leaves  all elements from~ \( V \) invariant:
$$ \left( \sum \alpha \, w_{1}  \cdots w_{n} \right)^{*} = \sum \alpha\, w_{n}  \cdots w_{1} , \quad \alpha \in \mathbb{R}, \quad  w_i \in V .$$

Define \( \|a\| = \text{tr}(a \cdot a^\ast). \) If \( a = \sum \alpha \, w_{1} \cdots w_{n} \), where  $\alpha \in \mathbb{R}$ and  $ w_i $ are arbitrary elements from $ V$, then
\[ \| a \| = \sum \alpha^2 \ f(w_{1}) \cdots f(w_{n}). \]
This  makes \( \mathcal{C}\ell(V,f) \) a normed algebra with \( \| a \cdot b \| \leq \| a \| \cdot \| b \| \).

Choose an orthonormal basis $ v_i,$ $ i \in \mathbb{N},$  in the vector space $V .$ As previously mentioned, choose an increasing sequence of even numbers \( 0 = n_0 < n_1 < n_2 < \cdots \). Let \( V_i , \) $i\in \mathbb{N},$ be the \( \mathbb{R} \)-linear span of \( v_{n_{i-1}+1}, \ldots, v_{n_i} \). Define
\[ c_i = v_{1} \cdots v_{n_i}, \quad A_1 = \mathcal{C}\ell(V_1,f),\quad A_i = \mathcal{C}\ell(V_i,f)_{\overline{0}} + c_i \, \mathcal{C}\ell(V_i,f)_{\overline{1}}, \quad  i \geq 2. \]
 As in Theorem \ref{TH_1} (though the field \( \mathbb{R} \) is not algebraically closed), 
\[\mathcal{C}\ell(V,f) \cong \bigotimes_{i \in \mathbb{N}} A_i, \quad A_i \cong \mathcal{C}\ell(V_i,f). \]
We identify an algebra $ A_i $ with the matrix algebra $ M_{m}(\mathbb{R}), $ where $ m = 2^{\frac{1}{2}(n_i - n_{i-1})}. $ Let $$ k = \frac{1}{2} m, \quad x_i = \begin{pmatrix} I_k & 0 \\ 0 & i\, I_k \end{pmatrix}, \quad a_i = \begin{pmatrix} 0 & I_k \\ 0 & 0 \end{pmatrix}, $$ 
where $ I_k $ is the identity matrix of order $ k.$

Let $ \varphi_i $ be the conjugation automorphism by   $x_i.$ Then $ \varphi_i(a_i) = i\, a_i. $ The sequence of automorphisms $ \varphi_1 \cdots \varphi_n,$ where $n \in \mathbb{N},$ converges in the Tykhonoff topology to an automorphism $ \varphi\in \text{Aut}(\mathcal{C}\ell(V,f)). $ Indeed, for an arbitrary element $ a \in \mathcal{C}\ell(V,f) $, the sequence $\varphi_1(\varphi_2(\ldots(\varphi_n(a))\ldots)) $, where  $ n \in \mathbb{N}, $ stabilizes. If $ a$ belongs to $ A_1 \otimes \ldots \otimes A_n ,$ then $ \varphi_1 \ldots \varphi_n(a) = \varphi_1 \ldots \varphi_m (a) $ for any $ m \geq n.$

The sequence $ \frac{1}{i} a_i $ converges to $ 0, $ whereas the sequence $ \varphi(\frac{1}{i} a_i) = \frac{1}{i} \varphi_i(a_i) = a_i $ does not.  Hence, the automorphism $ \varphi $ is not continuous.

 It would be interesting to construct such automorphisms of Clifford algebras $\mathcal{C}\ell(V,f),$   where the dimension of $ V $ is not countable, particularly for separable Hilbert spaces. In   \cite{BezOl}, we showed that the Clifford algebra of  a Hilbert space is not isomorphic to an infinite tensor product of matrix algebras.

\end{document}